\documentclass{amsart}

\usepackage{amsthm,amsmath,amsfonts,amssymb,shuffle}
\usepackage{tikz}
\usetikzlibrary{decorations.markings}
\usepackage{algorithm}
\usepackage{algpseudocode}

\newtheorem{theorem}{Theorem}[section]
\newtheorem{lemma}[theorem]{Lemma}
\newtheorem{proposition}[theorem]{Proposition}
\newtheorem{corollary}[theorem]{Corollary}
\newtheorem{conjecture}[theorem]{Conjecture}

\theoremstyle{definition}
\newtheorem{definition}[theorem]{Definition}
\newtheorem{example}{Example}

\theoremstyle{remark}
\newtheorem{remark}[theorem]{Remark}

\numberwithin{equation}{section}

\newcommand{\inv}{\ensuremath\mathrm{inv}}

\newcommand{\row}{\ensuremath\mathrm{row}}
\newcommand{\perm}{\ensuremath\mathrm{perm}}

\newcommand{\Red}{\ensuremath{R}}

\newcommand{\D}{\ensuremath\mathbb{D}}

\newcommand{\Bal}{\ensuremath\mathrm{SBT}}

\newcommand{\swap}{\ensuremath\mathfrak{c}}
\newcommand{\braid}{\ensuremath\mathfrak{b}}

\newcommand{\newword}[1]{\textbf{\emph{#1}}}

\newlength\cellsize 

\savebox2{%
\begin{picture}(12,12)
\put(0,0){\line(1,0){12}}
\put(0,0){\line(0,1){12}}
\put(12,0){\line(0,1){12}}
\put(0,12){\line(1,0){12}}
\end{picture}}

\newcommand\cellify[1]{\def\thearg{#1}\def\nothing{}%
\ifx\thearg\nothing\vrule width0pt height\cellsize depth0pt%
  \else\hbox to 0pt{\usebox2\hss}\fi%
  \vbox to 12\unitlength{\vss\hbox to 12\unitlength{\hss$#1$\hss}\vss}}

\newcommand\tableau[1]{\vtop{\let\\=\cr
\setlength\baselineskip{-12000pt}
\setlength\lineskiplimit{12000pt}
\setlength\lineskip{0pt}
\halign{&\cellify{##}\cr#1\crcr}}}

\newcommand{\gb}{\textcolor[RGB]{220,220,220}{\rule{1\cellsize}{1\cellsize}}\hspace{-\cellsize}\usebox2}

\newcommand\nocellify[1]{\def\thearg{#1}\def\nothing{}%
\ifx\thearg\nothing\vrule width0pt height\cellsize depth0pt%
  \else\hbox to 0pt{\hss}\fi%
  \vbox to 12\unitlength{\vss\hbox to 12\unitlength{\hss$#1$\hss}\vss}}

\newcommand\notableau[1]{\vtop{\let\\=\cr
\setlength\baselineskip{-12000pt}
\setlength\lineskiplimit{12000pt}
\setlength\lineskip{0pt}
\halign{&\nocellify{##}\cr#1\crcr}}}

\newlength\smcellsize 

\savebox3{%
\begin{picture}(9,9)
\put(0,0){\line(1,0){9}}
\put(0,0){\line(0,1){9}}
\put(9,0){\line(0,1){9}}
\put(0,9){\line(1,0){9}}
\end{picture}}

\newcommand\smcellify[1]{\def\thearg{#1}\def\nothing{}%
\ifx\thearg\nothing\vrule width0pt height\smcellsize depth0pt%
  \else\hbox to 0pt{\usebox3\hss}\fi%
  \vbox to \smcellsize{\vss\hbox to \smcellsize{\hss$_{#1}$\hss}\vss}}

\newcommand\smtab[1]{\vtop{\let\\=\cr
\setlength\baselineskip{-9000pt}
\setlength\lineskiplimit{9000pt}
\setlength\lineskip{0pt}
\halign{&\smcellify{##}\cr#1\crcr}}}

\savebox4{
\begin{picture}(10,10)
\put(5,5){\circle{10}}
\end{picture}}

\newcommand{\cir}[1]{\def\thearg{#1}\def\nothing{}%
\ifx\thearg\nothing\vrule width0pt height10\unitlength depth0pt%
  \else\hbox to 0pt{\usebox4\hss}\fi%
  \vbox to 10\unitlength{\vss\hbox to 10\unitlength{\hss$#1$\hss}\vss}}

\begin{document}


\title{An inversion metric for reduced words}  

\author{Sami Assaf}
\address{Department of Mathematics, University of Southern California, 3620 South Vermont Avenue, Los Angeles, CA 90089-2532, U.S.A.}
\email{shassaf@usc.edu}
\thanks{Work supported in part by the Simons Foundation (Award 524477, S.A.).}

\subjclass[2010]{%
  Primary 05A05, 05E18; %
  Secondary 05A15, 05A19}




\keywords{reduced words, balanced tableaux, inversion, Yang--Baxter moves}

\begin{abstract}
  We study the graph on reduced words with edges given by the Coxeter relations for the symmetric group. We define a metric on reduced words for a given permutation, analogous to Coxeter length for permutations, for which the graph becomes ranked with unique maximal element. We show this metric extends naturally to balanced tableaux, and use it to recover enumerative results of Edelman and Greene and of Reiner and Roichman.
\end{abstract}

\maketitle

%
\section{Introduction}
%
\label{sec:introduction}

The symmetric group $\mathfrak{S}_n$ has a Coxeter presentation with generators $s_i$, the simple transpositions interchanging $i$ and $i+1$, and Coxeter relations
\begin{enumerate}
\item $s_i s_j = s_j s_i$ for $|i-j| \geq 2$,
\item $s_i s_{i+1} s_i = s_{i+1} s_i s_{i+1}$ for $1\leq i \leq n-2$,
\end{enumerate}
and $s_i^2$ is the identity. We call (1) a \newword{commutation} and (2) a \newword{Yang--Baxter} move.

Given any permutation $w \in \mathfrak{S}_n$, a \newword{reduced word} for $w$ is a sequence $\rho = (\rho_{\ell(w)}, \ldots, \rho_1)$ such that $w = s_{\rho_{\ell(w)}} \cdots s_{\rho_1}$, where $\ell(w)$ is the \newword{length} of $w$ given by the number of pairs $(i<j)$ such that $w_i > w_j$. 

Tits \cite{Tit69} studied the graph with vertex set given by reduced words and edges connecting two reduced words that differ by a single Coxeter relation. In particular, he showed that the subgraph on reduced words for a given permutation is connected. There has been much research on this graph, in particular for reduced words for the longest permutation $w_0^{(n)}$ of $\mathfrak{S}_n$. In this paper, we add additional structure to this graph, making it into a ranked poset with canonical maximal element. From this we derive an explicit inversion metric on reduced words for the same permutation that precisely gives the minimum number of Coxeter relations needed to transform one into another, along with how many are commutations and how many Yang--Baxter moves. Dehornoy and Autord \cite{DA10} considered a similar question, phrased as computing the diameter of the graph on reduced words for $w_0^{(n)}$. They used techniques in group theory give a series of bounds and asymptotics, results which can be made explicit with this new metric.

Edelman and Greene \cite{EG87} introduced balanced tableaux to prove bijectively a result of Stanley \cite{Sta84} equating reduced words for $w_0^{(n)}$ with standard Young tableaux of staircase shape. The poset structure and inversion statistic extend naturally to balanced tableaux, where the constructions simplify greatly. We use this simplified metric on balanced tableaux to give a new, elementary proof of a result of Reiner and Roichman \cite{RR13} computing the diameter of the graph on reduced words for $w_0^{(n)}$. 

%

%
\section{Reduced words}
%
\label{sec:schubert}

Let $\Red(w)$ denote the set of reduced words for $w$, indexed \emph{from right to left} to mirror the action of $s_i$ as a function on permutations. 

\begin{example}[Reduced words]
  Take $w$ to be the permutation $42153$. Then the word $(\rho_5,\rho_4,\rho_3,\rho_2,\rho_1) = (1,4,2,3,1)$ is a reduced word for $w$ since
  \begin{displaymath}
    \begin{array}{rcr}
      s_1 s_4 s_2 s_3 s_1 & = & s_1 s_4 s_2 s_3 s_1 \cdot 12345 \\
      & = & s_1 s_4 s_2 s_3 \cdot 21345 \\
      & = & s_1 s_4 s_2 \cdot 21435 \\
      & = & s_1 s_4 \cdot 24135 \\
      & = & s_1 \cdot 24153 \\
      & = & 42153 
    \end{array}
  \end{displaymath}
  The $11$ reduced words in $\Red(42153)$ are shown in Fig.~\ref{fig:reduced}.
  \label{ex:reduced}
\end{example}

\begin{figure}[ht]
  \begin{displaymath}
    \begin{array}{c}
      (4,2,1,2,3) \hspace{1.5ex} (4,1,2,1,3) \hspace{1.5ex} (4,1,2,3,1) \hspace{1.5ex} (2,4,1,2,3) \hspace{1.5ex} (2,1,4,2,3) \hspace{1.5ex} (2,1,2,4,3) \\[1ex]
      (1,4,2,3,1) \hspace{1.5ex} (1,2,4,3,1) \hspace{1.5ex} (1,4,2,1,3) \hspace{1.5ex} (1,2,4,1,3) \hspace{1.5ex} (1,2,1,4,3)
    \end{array}
  \end{displaymath}
  \caption{\label{fig:reduced}The reduced words for $42153$.}
\end{figure}

\begin{remark}
  A pair of indices $(i<j)$ such that $w_i > w_j$ is called an \emph{inversion} of $w$, and the number of such pairs the \emph{inversion number} of $w$. We avoid this terminology here, instead referring to the latter as the \emph{length} of the permutation, in order to avoid confusion with the upcoming definition of \emph{inversions} for reduced words.
\end{remark}

\begin{definition}
  The \newword{run decomposition} of $\rho$, denoted by $(\rho^{(k)} | \cdots | \rho^{(1)})$, partitions $\rho$ into decreasing sequences (read from right to left) of maximal length.
\end{definition}

\begin{example}[Run decomposition]
  The word $\rho = (5,6,3,4,5,7,3,1,4,2,3,6)$, a reduced word for the permutation $w = 41758236$, has run decomposition
  \[ ( \overbrace{5 , 6}^{\rho^{(5)}} \mid \overbrace{3 , 4 , 5 , 7}^{\rho^{(4)}} \mid \overbrace{3}^{\rho^{(3)}} \mid \overbrace{1 , 4}^{\rho^{(2)}} \mid \overbrace{2 , 3 , 6}^{\rho^{(1)}} ) \] \vspace{-\baselineskip}
  \label{ex:run}
\end{example}

The following definition for \emph{super-Yamanouchi} words first appears in \cite{Ass-T}, where it is shown that the reduced word contributing the unique leading term to a Schubert polynomial is precisely this super-Yamanouchi word. The terminology derives from \emph{Yamanouchi} words, which capture the unique leading terms for Schur polynomials.

\begin{definition}
  A reduced word $\rho$ with run decomposition $(\rho^{(k)} | \cdots | \rho^{(1)})$ is \newword{super-Yamanouchi} if each $\rho^{(i)}$ is an interval and $\min(\rho^{(k)}) > \cdots > \min(\rho^{(1)})$.
  \label{def:re-super}
\end{definition}

\begin{example}[Super-Yamanouchi]
  The word $\rho = (5,6,3,4,5,7,3,1,4,2,3,6)$ from Example~\ref{ex:run} is \emph{not} super-Yamanouchi since none of $\rho^{(4)}, \rho^{(2)}, \rho^{(1)}$ is an interval, and since neither $\min(\rho^{(4)}) > \min(\rho^{(3)})$ nor $\min(\rho^{(2)}) > \min(\rho^{(1)})$ holds.

  In contrast, the word $\rho = (5,6,7,4,5,3,4,5,6,1,2,3)$, another reduced word for the same permutation, is super-Yamanouchi, with run decomposition
  \[ ( \overbrace{5 , 6, 7}^{\rho^{(4)}} \mid \overbrace{4 , 5 }^{\rho^{(3)}} \mid \overbrace{3, 4, 5, 6}^{\rho^{(2)}} \mid \overbrace{1 , 2, 3}^{\rho^{(1)}}  ) , \]
  so each run is an interval and $\min(\rho^{(4)}) > \min(\rho^{(3)}) > \min(\rho^{(2)}) > \min(\rho^{(1)})$.
  \label{ex:super-Y}
\end{example}

\begin{proposition}
  For any $w$, there exists a unique super-Yamanouchi $\pi\in\Red(w)$. 
  \label{prop:super-Y}
\end{proposition}

\begin{proof}
  Given $w$, construct $\pi$ according to Algorithm~\ref{alg:super-Y}. To see this is well-defined, the set in line 5 is nonempty whenever $\ell(v)>0$, the set in line 6 is nonempty by construction, and line 8 removes precisely $(j-2)-i+1 = j-i-1 \geq 1$ inversions from $v$, ensuring that algorithm terminates. Line 8 also ensures that the resulting word $\pi$ will be a word for $w$ and will be reduced since $(j-2)-i+1$ inversions are removed when appending $(j-2)-i+1$ letters to $\pi$. Each pass through line 7 appends an interval to $\pi$, so to check the super-Yamanouchi condition, we need only check that a subsequent pass chooses a smaller index at line 5. If $i$ is chosen in line 5, then after line 8 $v$ has no inversions weakly beyond index $i$, ensuring that the maximum in line 5 of the next iteration is strictly less than $i$. Therefore Algorithm~\ref{alg:super-Y} is well-defined and returns a super-Yamanouchi reduced word for $w$.

  Now suppose that $\rho \neq \pi$ is another super-Yamanouchi reduced word for $w$. Let $i$ be the maximum index for which $\pi_i \neq \rho_i$. Clearly removing the prefix or suffix of a reduced word does not change that it is reduced. Moreover, this also preserves the super-Yamanouchi property since runs must still form intervals and only the leftmost run can have a changed minimum, which necessarily gets weakly larger. Furthermore, removing the same prefix or suffix for two reduced words for the same permutation results again in (shorter) reduced words for the new permutation. Therefore by removing the suffix $\pi_{\ell},\pi_{\ell-1}\cdots\pi_{i+1}$ from both $\pi$ and $\rho$, we may assume $i = \ell$.

  The interval condition for super-Yamanouchi words ensures that a letter in position $i$ of $w$ is moved by success $s_k$'s to some position $j>i$, and the decreasing minimum condition for super-Yamanouchi words ensures that the subsequent letter moved is strictly left of position $i$. In order to be a reduced word, we must have $w_{\rho_{\ell}} > w_{\rho_{\ell}+1}$. Since $\pi$ is constructed by choosing the maximum $i$ such that $w_i > w_{i+1}$, we must have $\pi_{\ell} > \rho_{\ell}$. Since $\rho$ first selects an index $\rho_{\ell} < \pi_{\ell}$, and since each run of $\rho$ either fixes the position of the final descent or moves it one position to the left, based on whether or not that run crosses over the descent, there is no way to begin a new run with the final descent without violating the super-Yamanouchi condition. Thus $\pi$ is the unique super-Yamanouchi word for $w$.
\end{proof}

\begin{algorithm}[ht]
  \begin{algorithmic}[1]
    \Procedure{super}{$w$}
    \State $v \gets w$
    \State $\pi \gets ()$ 
    \While {$\ell(v)>0$}
      \State $i \gets \max\{ k \mid w_k > w_{k+1} \}$
      \State $j \gets \min\left\{ \{ k \mid w_i < w_k \} \cup \{n+1\} \right\}$
      \State $\pi \gets (\pi, i, i+1, \ldots, j-2)$
      \State $v \gets s_{j-2} \cdots s_{i+1} s_i v$
    \EndWhile
    \State \textbf{return} $\pi$
    \EndProcedure
  \end{algorithmic}
  \caption{\label{alg:super-Y}Super-Yamanouchi reduced word}
\end{algorithm}

\begin{example}[Super-Yamanouchi reduced word]
  Construct the super-Yamanouchi reduced word for the permutation $w = 41758236$ by Algorithm~\ref{alg:super-Y} as illustrated in Fig.~\ref{fig:super-Y}. We initialize with $v = 41758236$ and $\pi=()$, and then
  \begin{itemize}
  \item[loop 1:] $i=5$ and $j=8+1=9$, resulting in $\pi = (\mathbf{5,6,7})$ and now $v = 4175236\mathbf{8}$;
  \item[loop 2:] $i=4$ and $j=7$, and so $\pi = (5,6,7,\mathbf{4,5})$ and  $v = 41723\mathbf{5}68$;
  \item[loop 3:] $i=3$ and $j=8$, and so $\pi = (5,6,7,4,5,\mathbf{3,4,5,6})$ and $v = 412356\mathbf{7}8$;
  \item[loop 4:] $i=1$ and $j=5$, and so $\pi = (5,6,7,4,5,3,4,5,6,\mathbf{1,2,3})$ and $v = 123\mathbf{4}5678$. 
  \end{itemize}
  Having reached the identity, we terminate. Therefore the unique super-Yamanouchi reduced word for $w = 41758236$ is $\pi = (5,6,7,4,5,3,4,5,6,1,2,3)$.
  \label{ex:construct-super-Y}
\end{example}

\begin{figure}[ht]
  \begin{displaymath}
    4175\raisebox{-0.15\cellsize}{$\cir{\mathbf{8}}$}236 \xrightarrow{(5,6,7,}
    417 \raisebox{-0.15\cellsize}{$\cir{\mathbf{5}}$}2368 \xrightarrow{4,5,}   
    41  \raisebox{-0.15\cellsize}{$\cir{\mathbf{7}}$}23568 \xrightarrow{3,4,5,6,}
        \raisebox{-0.15\cellsize}{$\cir{\mathbf{4}}$}1235678 \xrightarrow{1,2,3)} 
    12345678 
  \end{displaymath}
  \caption{\label{fig:super-Y}An illustration of Algorithm~\ref{alg:super-Y} for the permutation $41758236$.}
\end{figure}

We define two involutions on reduced words for a given permutation based on the Coxeter relations for the simple transpositions.

\begin{definition}
  Given $w$ and $1 \leq i < \ell(w)$, $\swap_i$ acts on $\rho \in \Red(w)$ by \emph{commuting} $\rho_i$ and $\rho_{i+1}$ whenever $|\rho_i - \rho_{i+1}| > 1$ and the identity otherwise.
  \label{def:rex-swap}
\end{definition}

\begin{definition}
  Given $w$ and $1 < i < \ell(w)$,  $\braid_i$ acts on $\rho \in \Red(w)$ by \emph{braiding} $\rho_{i-1} \rho_i \rho_{i+1}$ to $\rho_{i} \rho_{i \pm 1} \rho_{i}$ whenever $\rho_{i-1} = \rho_{i+1} = \rho_i \pm 1$ and the identity otherwise.
  \label{def:rex-braid}
\end{definition}

\begin{figure}[ht]
  \begin{center}
    \begin{tikzpicture}[xscale=1.75,yscale=1.25,
        label/.style={%
          postaction={ decorate,
            decoration={ markings, mark=at position 0.75 with \node #1;}}}]
      \node at (3,5) (C5) {$(4,2,1,2,3)$};
      \node at (2,4) (B4) {$(2,4,1,2,3)$};
      \node at (4,4) (D4) {$(4,1,2,1,3)$};
      \node at (1,3) (A3) {$(2,1,4,2,3)$};
      \node at (3,3) (C3) {$(1,4,2,1,3)$};
      \node at (5,3) (E3) {$(4,1,2,3,1)$};
      \node at (0,2) (A2) {$(2,1,2,4,3)$};
      \node at (2,2) (C2) {$(1,2,4,1,3)$};
      \node at (4,2) (E2) {$(1,4,2,3,1)$};
      \node at (1,1) (B1) {$(1,2,1,4,3)$};
      \node at (3,1) (D1) {$(1,2,4,3,1)$};
      \draw[thin,label={[above]{$\swap_4$}}] (C5) -- (B4) ;
      \draw[thin,label={[above]{$\braid_3$}}](C5) -- (D4) ;
      \draw[thin,label={[above]{$\swap_3$}}] (B4) -- (A3) ;
      \draw[thin,label={[above]{$\swap_4$}}] (D4) -- (C3) ;
      \draw[thin,label={[above]{$\swap_1$}}] (D4) -- (E3) ;
      \draw[thin,label={[above]{$\swap_2$}}] (A3) -- (A2) ;
      \draw[thin,label={[above]{$\swap_3$}}] (C3) -- (C2) ;
      \draw[thin,label={[above]{$\swap_1$}}] (C3) -- (E2) ;
      \draw[thin,label={[above]{$\swap_4$}}] (E3) -- (E2) ;
      \draw[thin,label={[above]{$\braid_4$}}](A2) -- (B1) ;
      \draw[thin,label={[above]{$\swap_2$}}] (C2) -- (B1) ;
      \draw[thin,label={[above]{$\swap_1$}}] (C2) -- (D1) ;
      \draw[thin,label={[above]{$\swap_3$}}] (E2) -- (D1) ;
    \end{tikzpicture}
  \caption{\label{fig:rex-relations}An illustration of the Coxeter moves on $\Red(42153)$.}
  \end{center}
\end{figure}

We refer to $\swap_i$ as a commutation, to $\braid_i$ as a Yang--Baxter move, and to either as a Coxeter move. For examples of Coxeter moves on reduced words, see Fig.~\ref{fig:rex-relations}.

It follows from classical work of Tits \cite{Tit69} that the maps $\swap_i$ and $\braid_i$ are well-defined involutions on $\Red(w)$, and that the graph on $\Red(w)$ with edges given by $\swap_i$ and $\braid_i$ is connected. Pushing this further, Fig.~\ref{fig:rex-relations} suggests a ranked poset structure on reduced words for $w$  with unique maximal element equal to the super-Yamanouchi reduced word for $w$. The following definition measures the \emph{minimum} number of commutations and Yang--Baxter moves needed to get from a given reduced word to the super-Yamanouchi one.

\begin{definition}
  Given $\rho \in \Red(w)$, define the \newword{inversion number of $\rho$} by
  \begin{equation}
    \inv(\rho) = \ell(v(\rho)) - \sum_i \left( \pi_i - \rho_{i} \right),
  \end{equation}
  where $\pi \in \Red(w)$ is super-Yamanouchi and $v(\rho)$ is defined by Algorithm~\ref{alg:rex-perm}.
  \label{def:rex-inv}
\end{definition}

\begin{algorithm}[ht]
  \begin{algorithmic}[1]
    \Procedure{perm}{$\rho$}
    \State $\pi\gets $ super-Yamanouchi reduced word for $w$ 
    \State $\perm\gets $ identity permutation of $\mathfrak{S}_{\ell(w)}$ 
    \For{$i$ from $\ell(w)$ to $1$ by $-1$} 
    \State $k \gets \pi_i$ 
    \For{$j$ from $\ell(w)$ to $1$ by $-1$} 
    \If{$\rho_j = k$ and is not already paired} 
    \State pair $\rho_j$ with $\pi_i$
    \State $\perm_i \gets j$
    \State \textbf{break}
    \ElsIf{$\rho_j = k-1$ and is not already paired} 
    \State $k \gets k-1$
    \State \textbf{next}
    \EndIf
    \EndFor
    \EndFor
    \State \textbf{return} $\perm$
    \EndProcedure
  \end{algorithmic}
  \caption{\label{alg:rex-perm}Permutation of a reduced word}
\end{algorithm}

\begin{example}[Inversions for reduced words]
  Let $\rho = (5,6,3,4,5,7,3,1,4,2,3,6)$. The super-Yamanouchi reduced word is $\pi = (5,6,7,4,5,3,4,5,6,1,2,3)$. Following Algorithm~\ref{alg:rex-perm}, the first three iterations of the \textbf{for} loop on line 4 ($i=12,11,10$) will be satisfied by the \textbf{if} condition of line 7, resulting in $\pi_{12}=5$, $\pi_{11}=6$, $\pi_{10}=7$ paired with $\rho_{12} = 5$, $\rho_{11} = 6$, $\rho_{7} = 7$, respectively.

  On the fourth iteration of the \textbf{for} loop on line 4 ($i=9$), we set $k= \pi_9 = 4$ on line 5, and on the third iteration of the \textbf{for} loop on line 6 ($j=10$), the \textbf{else if} condition on line 11 is met, and we decrement $k=3$. Then, on the seventh iteration of the \textbf{for} loop on line 6 ($j=6$), the \textbf{if} condition of line 7 is met and we pair $\pi_{9}=4$ with $\rho_{6}=3$. Continuing thus, we pair values of $\pi$ from left to right with values of $\rho$ as illustrated in Fig.~\ref{fig:rex-inv}. 

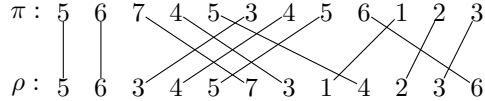
\begin{figure}[ht]
  \begin{center}
    \begin{tikzpicture}[every node/.style={inner sep=0pt},scale=0.5]
      \node at (0,2)  (P0) {$\pi:$};
      \node at (1,2)  (P12) {$5$};
      \node at (2,2)  (P11) {$6$};
      \node at (3,2)  (P10) {$7$};
      \node at (4,2)  (P9)  {$4$};
      \node at (5,2)  (P8)  {$5$};
      \node at (6,2)  (P7)  {$3$};
      \node at (7,2)  (P6)  {$4$};
      \node at (8,2)  (P5)  {$5$};
      \node at (9,2)  (P4)  {$6$};
      \node at (10,2) (P3)  {$1$};
      \node at (11,2) (P2)  {$2$};
      \node at (12,2) (P1)  {$3$};
      \node at (0,0)  (R0) {$\rho:$};
      \node at (1,0)  (R12) {$5$};
      \node at (2,0)  (R11) {$6$};
      \node at (3,0)  (R10) {$3$};
      \node at (4,0)  (R9)  {$4$};
      \node at (5,0)  (R8)  {$5$};
      \node at (6,0)  (R7)  {$7$};
      \node at (7,0)  (R6)  {$3$};
      \node at (8,0)  (R5)  {$1$};
      \node at (9,0)  (R4)  {$4$};
      \node at (10,0) (R3)  {$2$};
      \node at (11,0) (R2)  {$3$};
      \node at (12,0) (R1)  {$6$};
      \draw  (P1) -- (R2) ;
      \draw  (P2) -- (R3) ;
      \draw  (P3) -- (R5) ;
      \draw  (P4) -- (R1) ;
      \draw  (P5) -- (R8) ;
      \draw  (P6) -- (R9) ;
      \draw  (P7) -- (R10);
      \draw  (P8) -- (R4) ;
      \draw  (P9) -- (R6) ;
      \draw (P10) -- (R7) ;
      \draw (P11) -- (R11);
      \draw (P12) -- (R12);
    \end{tikzpicture}
    \caption{\label{fig:rex-inv}An illustration of the pairings in Algorithm~\ref{alg:rex-perm} for the reduced word $\rho = (5,6,3,4,5,7,3,1,4,2,3,6)$.}
  \end{center}
\end{figure}

  Therefore $\perm(\rho) = 2 \, 3 \, 5 \, 1 \, 8 \, 9 \, 1\!0 \, 4 \, 6 \, 7 \, 1\!1 \, 1\!2$ and so $\inv(\rho) = 13-2 = 11$. Note
  \[ \rho = \swap_7 \, \swap_8 \, \swap_9 \, \swap_4 \, \swap_6 \, \braid_8 \, \braid_6 \, \swap_7 \, \swap_1 \, \swap_2 \, \swap_3 \, \pi, \]
  which is a sequence of $11$ involutions, exactly $2$ of which are Yang--Baxter moves.
  \label{ex:inversion-R}
\end{example}

\begin{theorem}
  For $\rho \in \Red(w)$, $\inv(\rho)$ is a well-defined non-negative integer. Moreover, there exists a sequence $f = f_{\inv(\rho)} \cdots f_1$ of Coxeter moves, i.e. $f_j = \swap_i$ or $\braid_i$, such that $f(\rho)$ is super-Yamanouchi, and for any sequence $g = g_m \cdots g_1$ of Coxeter moves such that $g(\rho)$ is super-Yamanouchi, we have $m \geq \inv(\rho)$.
  \label{thm:rex-super}
\end{theorem}

\begin{proof}
  We claim the theorem holds for $\rho$ if and only if it holds for $\swap_i(\rho)$. This is vacuously true if $\swap_i$ acts trivially on $\rho$. Otherwise, $\swap_i(\rho)$ will have permutation $s_i \perm(\rho)$, and, since the letters of $\rho$ and $\swap_i(\rho)$ are the same, we have
  \begin{eqnarray*}
    \inv(\swap_i \rho) & = & \inv (s_i \perm(\rho)) - \sum \left( \pi_j - (\swap_i\rho)_{j} \right) \\
    & = & \inv(\perm(\rho)) \pm 1 - \sum \left( \pi_j - \rho_{j} \right) = \inv(\rho) \pm 1.
  \end{eqnarray*}
  Furthermore, $\inv(\swap_i \rho) = \inv(\rho)+1$ precisely when $i$ is left of $i+1$ in $\perm(\rho)$. 

  Next we claim the theorem holds for $\rho$ if and only if it holds for $\braid_i(\rho)$. If $\braid_i$ acts trivially on $\rho$, the claim is vacuously true. Otherwise, $\braid_i(\rho)$ will have permutation $s_i s_{i-1} \perm(\rho)$ or $s_{i-1} s_{i} \perm(\rho)$, the former when $\rho_{i\pm 1} = \rho_i +1$ and the latter when $\rho_{i\pm 1} = \rho_i -1$. Assuming the former, we have
  \begin{eqnarray*}
    \inv(\braid_i \rho) & = & \inv (s_{i-1} s_i \perm(\rho)) - \sum \left( \pi_j - (\braid_i\rho)_{j} \right) \\
    & = & \inv(\perm(\rho)) + 2 - \big( \sum \left( \pi_j - \rho_{j} \right) + 1 \big) = \inv(\rho) + 1,
  \end{eqnarray*}
  and, by the same computation, $\inv(\braid_i \rho) = \inv(\rho)-1$ in the latter case.

  Recall from earlier that any two reduced words for $w$ can be transformed into one another by a sequence of Coxeter moves. Let $m$ be the minimum number of Coxeter moves needed to transform $\rho$ into the super-Yamanouchi reduced word. If $m=0$, then $\rho$ is super-Yamanouchi, in which case the permutation for $\rho$ is the identity and $\inv(\rho) = 0$, so the theorem holds. Assume, for induction, that the theorem holds for any $n<m$, and suppose $\rho = f_m \cdots f_1 \pi$, where $\pi$ is super-Yamanouchi and $f_j$ is $\swap_i$ or $\braid_i$ for some $i$. By induction, the result holds for $f_{m-1} \cdots f_1 \pi = f_m \rho$, and so, by the claims, it holds for $\rho$ as well. 
\end{proof}

Thus we may define the \newword{inversion poset for reduced words} as follows.

\begin{corollary}
  For $w$ a permutation, the partial order on $\Red(w)$ given by the transitive closure of covering relations
  \begin{itemize}
  \item $\rho > \swap_i\rho$ if $\inv(\swap_i\rho) = \inv(\rho)+1$, and
  \item $\rho > \braid_i\rho$ if $\inv(\braid_i\rho) = \inv(\rho)+1$
  \end{itemize}
  makes $\Red(w)$ into a ranked partially ordered set with unique maximal element.
  \label{cor:rex-poset}
\end{corollary}

Notice the ranking is the \emph{co-inversion number}, so that the super-Yamanouchi word is the unique \emph{maximal} element in line with convention from Schubert calculus.

From the proof of Theorem~\ref{thm:rex-super}, can, in fact, count the minimum number of \emph{Yang--Baxter moves} on any shortest path from a reduced word to the super-Yamanouchi reduced word by considering the offset between the length of the permutation of $\rho$ and the inversion number of $\rho$. More generally, we have the following. 

\begin{corollary}
  For $\rho,\sigma \in \Red(w)$, and $f = f_{k} \cdots f_1$ any minimal length sequence of Coxeter moves, i.e. $f_j = \swap_i$ or $\braid_i$, such that $f(\rho) = \sigma$, the number of Coxeter moves that are Yang--Baxter moves is given by
  \begin{equation}
    \#\{ j \mid f_j = \braid_i \mbox{ for some } i\} = \sum_i \left| \rho_i - \sigma_{\left(\perm(\sigma) \perm(\rho)^{-1}\right)_i} \right| .
  \end{equation}
  \label{cor:yang-baxter}
\end{corollary}

While one can hope to define an explicit metric on reduced words analogous to Kendall's $\tau$ metric on permutations \cite{Ken38} by 
\begin{equation}
  \inv(\rho,\sigma) = \ell(\perm(\rho,\sigma)) - \sum_i \left| \rho_i - \sigma_{\perm(\rho,\sigma)_i} \right|,
\end{equation}
where $\perm(\rho,\sigma) = \perm(\sigma) \perm(\rho)^{-1}$, this does not always give the correct minimum distance between arbitrary reduced words. 

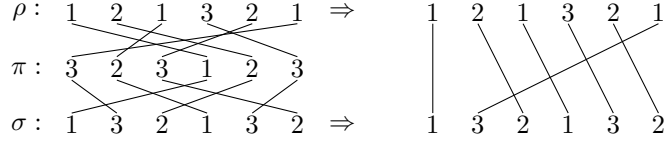
\begin{figure}[ht]
  \begin{center}
    \begin{tikzpicture}[every node/.style={inner sep=1pt},xscale=0.6,yscale=0.5]
      \node at (0,3)  (R0) {$\rho:$};
      \node at (1,3)  (R6) {$1$};
      \node at (2,3)  (R5) {$2$};
      \node at (3,3)  (R4) {$1$};
      \node at (4,3)  (R3) {$3$};
      \node at (5,3)  (R2) {$2$};
      \node at (6,3)  (R1) {$1$};
      \node at (0,1.5)  (P0) {$\pi:$};
      \node at (1,1.5)  (P6) {$3$};
      \node at (2,1.5)  (P5) {$2$};
      \node at (3,1.5)  (P4) {$3$};
      \node at (4,1.5)  (P3) {$1$};
      \node at (5,1.5)  (P2) {$2$};
      \node at (6,1.5)  (P1) {$3$};
      \node at (0,0)  (S0) {$\sigma:$};
      \node at (1,0)  (S6) {$1$};
      \node at (2,0)  (S5) {$3$};
      \node at (3,0)  (S4) {$2$};
      \node at (4,0)  (S3) {$1$};
      \node at (5,0)  (S2) {$3$};
      \node at (6,0)  (S1) {$2$};
      \draw  (P1.north) -- (R3.south) ;
      \draw  (P2.north) -- (R5.south) ;
      \draw  (P3.north) -- (R6.south) ;
      \draw  (P4.north) -- (R2.south) ;
      \draw  (P5.north) -- (R4.south) ;
      \draw  (P6.north) -- (R1.south) ;
      \draw  (P1.south) -- (S2.north) ;
      \draw  (P2.south) -- (S4.north) ;
      \draw  (P3.south) -- (S6.north) ;
      \draw  (P4.south) -- (S1.north) ;
      \draw  (P5.south) -- (S3.north) ;
      \draw  (P6.south) -- (S5.north) ;
      \node at (7,3) {$\Rightarrow$};
      \node at (9,3)  (RR6) {$1$};
      \node at (10,3) (RR5) {$2$};
      \node at (11,3) (RR4) {$1$};
      \node at (12,3) (RR3) {$3$};
      \node at (13,3) (RR2) {$2$};
      \node at (14,3) (RR1) {$1$};
      \node at (7,0) {$\Rightarrow$};
      \node at (9,0)  (SS6) {$1$};
      \node at (10,0) (SS5) {$3$};
      \node at (11,0) (SS4) {$2$};
      \node at (12,0) (SS3) {$1$};
      \node at (13,0) (SS2) {$3$};
      \node at (14,0) (SS1) {$2$};
      \draw  (SS2.north) -- (RR3.south) ;
      \draw  (SS4.north) -- (RR5.south) ;
      \draw  (SS6.north) -- (RR6.south) ;
      \draw  (SS1.north) -- (RR2.south) ;
      \draw  (SS3.north) -- (RR4.south) ;
      \draw  (SS5.north) -- (RR1.south) ;
    \end{tikzpicture}
    \caption{\label{fig:rex-metric}An illustration of the permutation of a pair of reduced words for $w = 42153$. Note this does not measure distance.}
  \end{center}
\end{figure}

\begin{example}[Barrier to inversion metric on reduced words]
  Let $\rho = (1,2,1,3,2,1)$ and $\sigma = (1,3,2,1,3,2)$, both reduced words for the long permutation $w_0^{(4)}=4321$. Then $\pi = (3,2,3,1,2,3)$ is the super-Yamanouchi word, and following Algorithm~\ref{alg:rex-perm}, we have the two pairings indicated on the left side of Fig.~\ref{fig:rex-metric}. Composing the diagram gives $\perm(\rho,\sigma) = 51234$, and so we have
  \[ \inv(\rho,\sigma) = \ell(51234) - |1-1| - |2-2| - |1-1| - |3-3| - |2-2| - |1-3| = 4-2 = 2. \]
  Observe, from Fig.~\ref{fig:R321}, any shortest path from $\rho$ to $\sigma$ has length $4$ and uses exactly $2$ Yang--Baxter moves. Thus the naive inversion number for arbitrary pairs does not work to give the correct minimum distance.
\end{example}

\begin{figure}[ht]
  \begin{center}
    \begin{tikzpicture}[xscale=1.45,yscale=1,
        label/.style={%
          postaction={ decorate,
            decoration={ markings, mark=at position 0.75 with \node #1;}}}]
      \node at (3,7) (R37) {$(3,2,3,1,2,3)$};
      \node at (2,6) (R26) {$(2,3,2,1,2,3)$};
      \node at (4,6) (R46) {$(3,2,1,3,2,3)$};
      \node at (1,5) (R15) {$(2,3,1,2,1,3)$};
      \node at (5,5) (R55) {$(3,2,1,2,3,2)$};
      \node at (0,4) (R04) {$(2,3,1,2,3,1)$};
      \node at (2,4) (R24) {$(2,1,3,2,1,3)$};
      \node at (6,4) (R64) {$(3,1,2,1,3,2)$};
      \node at (1,3) (R13) {$(2,1,3,2,3,1)$};
      \node at (5,3) (R53) {$(1,3,2,1,3,2)$};
      \node at (7,3) (R73) {$(3,1,2,3,1,2)$};
      \node at (2,2) (R22) {$(2,1,2,3,2,1)$};
      \node at (6,2) (R62) {$(1,3,2,3,1,2)$};
      \node at (3,1) (R31) {$(1,2,1,3,2,1)$};
      \node at (5,1) (R51) {$(1,2,3,2,1,2)$};
      \node at (4,0) (R40) {$(1,2,3,1,2,1)$};
      \draw[thin,label={[above]{$\braid_5$}}](R37) -- (R26) ;
      \draw[thin,label={[above]{$\swap_3$}}] (R37) -- (R46) ;
      \draw[thin,label={[above]{$\braid_3$}}](R26) -- (R15) ;
      \draw[thin,label={[above]{$\braid_2$}}](R46) -- (R55) ;
      \draw[thin,label={[above]{$\swap_1$}}] (R15) -- (R04) ;
      \draw[thin,label={[above]{$\swap_4$}}] (R15) -- (R24) ;
      \draw[thin,label={[above]{$\braid_4$}}](R55) -- (R64) ;
      \draw[thin,label={[above]{$\swap_4$}}] (R04) -- (R13) ;
      \draw[thin,label={[above]{$\swap_1$}}] (R24) -- (R13) ;
      \draw[thin,label={[above]{$\swap_5$}}] (R64) -- (R53) ;
      \draw[thin,label={[above]{$\swap_2$}}] (R64) -- (R73) ;
      \draw[thin,label={[above]{$\braid_3$}}](R13) -- (R22) ;
      \draw[thin,label={[above]{$\swap_2$}}] (R53) -- (R62) ;
      \draw[thin,label={[above]{$\swap_5$}}] (R73) -- (R62) ;
      \draw[thin,label={[above]{$\braid_5$}}](R22) -- (R31) ;
      \draw[thin,label={[above]{$\braid_4$}}](R62) -- (R51) ;
      \draw[thin,label={[above]{$\swap_3$}}] (R31) -- (R40) ;
      \draw[thin,label={[above]{$\braid_2$}}](R51) -- (R40) ;
    \end{tikzpicture}
  \caption{\label{fig:R321}An illustration of the Coxeter moves on $\Red(4321)$.}
  \end{center}
\end{figure}

%
\section{Balanced tableaux}
%
\label{sec:balanced}

The calculation of the inversion number for a reduced word is admittedly complicated, made more so by the requirement that one first compute the super-Yamanouchi word. By shifting our paradigm to another model for reduced words, this statistic becomes more natural and much simpler to compute.

The \newword{Rothe diagram} (also called the \newword{inversion diagram}) of a permutation $w$, denoted by $\D(w)$, is the following subset of cells in the first quadrant of the plane,
\begin{equation}
  \D(w) = \{ (i,w_j) \mid i<j \mbox{ and } w_i > w_j \} \subset \mathbb{Z}^{+}\times\mathbb{Z}^{+}.
  \label{e:rothe}
\end{equation}
The Rothe diagram of $w$ gives a graphical representation of the inversion pairs of $w$. In particular, the number of cells in $\D(w)$ is simply $\ell(w)$. 

\begin{example}[Rothe diagram]
  To draw the Rothe diagram for $w=41758236$, we write $w$ vertically along the $y$-axis with $w_i$ at height $i$, and we label the cells horizontally along the $x$-axis with positive integers, as illustrated in Fig.~\ref{fig:Rothe}. When computing the cells in row $3$, for instance, we consider $w_3=7$ and place cells in columns $5, 2, 3, 6$ since these occur to the right of and are smaller than $7$.
\end{example}

\begin{figure}[ht]
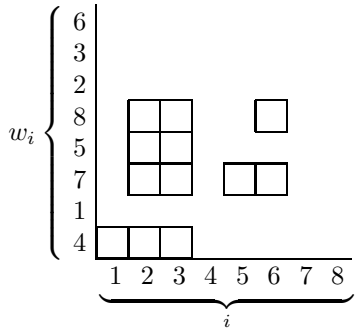

  \begin{displaymath}
    \begin{array}{r}
      w_i \left\{  
      \raisebox{3.25\cellsize}{$%
      \notableau{6 \\ 3 \\ 2 \\ 8 \\ 5 \\ 7 \\ 1 \\ 4}
      \vline\tableau{ \\ \\ \\ & \ & \ & & & \ & & \\ & \ & \ \\ & \ & \ & & \ & \ \\ \\ \ & \ & \ \\\hline }$}
      \right. \\
      \underbrace{\notableau{1 & 2 & 3 & 4 & 5 & 6 & 7 & 8}}_i
    \end{array}
  \end{displaymath}
  \caption{\label{fig:Rothe}The Rothe diagram $\D(w)$ for $w=41758236$.}
\end{figure}

The Rothe diagram of $w$ provides an alternative method from that described in Proposition~\ref{prop:super-Y} for computing the super-Yamanouchi reduced word for $w$.

\begin{definition}
  For $w$ a permutation, the \newword{row-interval filling} for $D(w)$ is the positive integer filling with entries $i, i+1, i+2, \ldots$ in row $i$, from left to right. 
\end{definition}

\begin{example}[Row-interval filling]
  The row-interval filling for $\D(41758236)$ is shown in Fig.~\ref{fig:row}. Comparing with Ex.~\ref{ex:super-Y}, notice that the row reading word of this filling, i.e. the word obtained by reading the rows from left to right beginning with the highest, is precisely the super-Yamanouchi word for $w$.
\end{example}

\begin{figure}[ht]
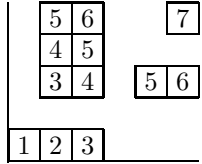

  \begin{displaymath}
    \vline\tableau{ & 5 & 6 & & & 7 \\ & 4 & 5 \\ & 3 & 4 & & 5 & 6 \\ \\ 1 & 2 & 3 \\\hline }
  \end{displaymath}
  \caption{\label{fig:row}The row-interval filling of $\D(41758236)$.}
\end{figure}

\begin{proposition}
  The row reading word of the row-interval filling for $w$ is precisely the super-Yamanouchi reduced word for $w$.
  \label{prop:super-bal}
\end{proposition}

\begin{proof}
  Following the procedure for computing $\pi$ in Algorithm~\ref{alg:super-Y}, the last descent of $w$ corresponds to the highest occupied row of $\D(w)$, and the number of positions the letter at that position must move to the right is precisely the number of entries in that row. Thus removing the final descent corresponds to removing the highest occupied row, and the same values are recorded for both processes.
\end{proof}

While this construction applies equally well to any diagram, for a Rothe diagram the columns will be integer intervals as well. In fact, this property uniquely characterizes diagrams as Rothe diagrams.

\begin{proposition}
  A cell diagram $D$ in the first quadrant is the Rothe diagram of a permutation if and only if the columns of the row-interval filling form increasing intervals from bottom to top, beginning with $i$ at the bottom of column $i$.
  \label{prop:rothe}
\end{proposition}

\begin{proof}
  From \eqref{e:rothe}, one sees that the Rothe diagram for $w^{-1}$ is the transpose of the Rothe diagram for $w$. Moreover, transposing the row-interval filling for $w$ results in the row-interval filling for $w^{-1}$, so the columns must form intervals as well. 
\end{proof}

Stanley \cite{Sta84} introduced a new family of symmetric functions indexed by permutations in order to enumerate reduced words. Edelman and Greene \cite{EG87} introduced balanced labelings of Rothe diagrams in order to prove Stanley's conjecture that his symmetric functions are Schur positive and to give a precise enumeration of reduced words. We review balanced tableaux here, but give independent, elementary proofs of their bijection with reduced words using the ranked poset structure. 

\begin{definition}[\cite{EG87}]
  A \newword{standard balanced tableau} is a bijective filling of a Rothe diagram with entries from $\{1,2,\ldots,n\}$ such that for every entry of the diagram, the number of entries to its right that are greater is equal to the number of entries above it that are smaller. 
  \label{def:balanced-tableau}
\end{definition}

Denote the set of standard balanced tableaux on $\D(w)$ by $\Bal(w)$. 

\begin{example}[Balanced tableaux]
  For $w = 42153$, the filling of $\D(w)$ on the left of Fig.~\ref{fig:hooks} is balanced since for each cell (indicated in bold), the cells above and to the right have the same number of entries above that are greater (indicated in circles) as entries to the right that are smaller (also indicated in circles).

\begin{figure}[ht]
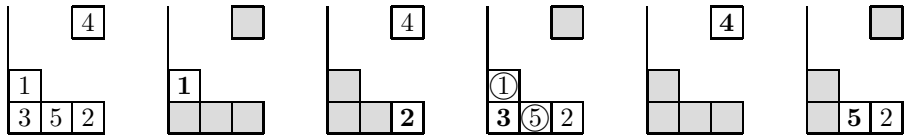
  
  \begin{displaymath}
    \begin{array}{c@{\hskip 2\cellsize}c@{\hskip 2\cellsize}c@{\hskip 2\cellsize}c@{\hskip 2\cellsize}c@{\hskip 2\cellsize}c}
      \vline\tableau{ & & 4 \\ \\ 1 \\ 3 & 5 & 2} &
      \vline\tableau{ & & \gb \\ \\ \mathbf{1} \\ \gb & \gb & \gb} &
      \vline\tableau{ & & 4 \\ \\ \gb \\ \gb & \gb & \mathbf{2}} &
      \vline\tableau{ & & \gb \\ \\ \cir{1} \\ \mathbf{3} & \cir{5} & 2} &
      \vline\tableau{ & & \mathbf{4} \\ \\ \gb \\ \gb & \gb & \gb} &
      \vline\tableau{ & & \gb \\ \\ \gb \\ \gb & \mathbf{5} & 2}
    \end{array}
  \end{displaymath}
  \caption{\label{fig:hooks}Checking the balanced condition for a standard tableaux.}
\end{figure}
  
  The $11$ balanced tableaux in $\Bal(42153)$ are shown in Fig.~\ref{fig:balanced}.
\end{example}

\begin{figure}[ht]
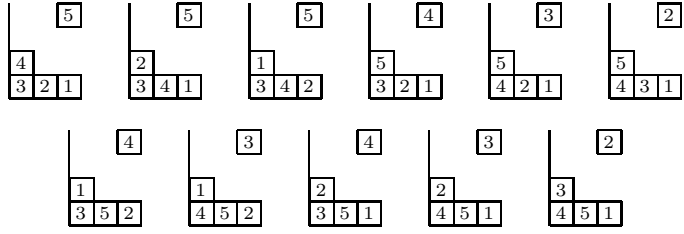

  \begin{displaymath}
    \begin{array}{c}
      \vline\smtab{ & & 5 \\ \\ 4 \\ 3 & 2 & 1 \\} \hspace{1.5\cellsize} 
      \vline\smtab{ & & 5 \\ \\ 2 \\ 3 & 4 & 1 \\} \hspace{1.5\cellsize}
      \vline\smtab{ & & 5 \\ \\ 1 \\ 3 & 4 & 2 \\} \hspace{1.5\cellsize}
      \vline\smtab{ & & 4 \\ \\ 5 \\ 3 & 2 & 1 \\} \hspace{1.5\cellsize}
      \vline\smtab{ & & 3 \\ \\ 5 \\ 4 & 2 & 1 \\} \hspace{1.5\cellsize}
      \vline\smtab{ & & 2 \\ \\ 5 \\ 4 & 3 & 1 \\} \\ \\
      \vline\smtab{ & & 4 \\ \\ 1 \\ 3 & 5 & 2 \\} \hspace{1.5\cellsize}
      \vline\smtab{ & & 3 \\ \\ 1 \\ 4 & 5 & 2 \\} \hspace{1.5\cellsize}
      \vline\smtab{ & & 4 \\ \\ 2 \\ 3 & 5 & 1 \\} \hspace{1.5\cellsize}
      \vline\smtab{ & & 3 \\ \\ 2 \\ 4 & 5 & 1 \\} \hspace{1.5\cellsize}
      \vline\smtab{ & & 2 \\ \\ 3 \\ 4 & 5 & 1 \\} 
    \end{array}
  \end{displaymath}
  \caption{\label{fig:balanced}The standard balanced tableaux for $42153$.}
\end{figure}

To prove standard balanced tableaux are in bijection with reduced words, first observe there is a canonical \emph{super-Yamanouchi} standard balanced tableau.

\begin{definition}
  A standard balanced tableau $R$ is \newword{super-Yamanouchi} if its reverse row reading word (right to left from bottom to top) is the identity.
  \label{def:bal-super}
\end{definition}

The balanced condition is immediate for the super-Yamanouchi tableau since entries increase in columns from bottom to top and in rows from left to right. For example, the super-Yamanouchi balanced tableau for $41758236$ is shown in Fig.~\ref{fig:super-B}.

\begin{figure}[ht]
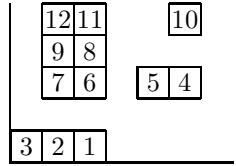

  \begin{displaymath}
    \vline\tableau{ & 12 & 11 & & & 10 \\ & 9 & 8 \\ & 7 & 6 & & 5 & 4 & \\ \\ 3 & 2 & 1 \\\hline }
  \end{displaymath}
  \caption{\label{fig:super-B}The super-Yamanouchi balanced tableau for $\D(w)$.}
\end{figure}

We next define simple analogs of the Coxeter moves for balanced tableaux, where the commutations involve two consecutive values and the Yang--Baxter moves involve three consecutive values. Both act only in certain circumstances.

\begin{definition}
  Given $w$ and $1 \leq i < \inv(w)$, $\swap_i$ acts on $\Bal(w)$ by exchanging $i$ and $i+1$ if they are not in the same row or column and by the identity otherwise.
  \label{def:bal-swap}
\end{definition}

\begin{definition}
  Given $w$ and $1 < i < \inv(w)$, $\braid_i$ acts on $\Bal(w)$ by exchanging $i-1$ and $i+1$ if one is in the same column and above $i$ and the other is in the same row and right of $i$ and by the identity otherwise.
  \label{def:bal-braid}
\end{definition}

\begin{figure}[ht]
  \begin{center}
    \begin{tikzpicture}[xscale=2,yscale=1.5,
        label/.style={%
          postaction={ decorate,
            decoration={ markings, mark=at position 0.75 with \node #1;}}}]
      \node at (3,5) (C5) {$\vline\smtab{& & 5 \\ \\ 4 \\ 3 & 2 & 1}$};
      \node at (2,4) (B4) {$\vline\smtab{& & 4 \\ \\ 5 \\ 3 & 2 & 1}$};
      \node at (4,4) (D4) {$\vline\smtab{& & 5 \\ \\ 2 \\ 3 & 4 & 1}$};
      \node at (1,3) (A3) {$\vline\smtab{& & 3 \\ \\ 5 \\ 4 & 2 & 1}$};
      \node at (3,3) (C3) {$\vline\smtab{& & 4 \\ \\ 2 \\ 3 & 5 & 1}$};
      \node at (5,3) (E3) {$\vline\smtab{& & 5 \\ \\ 1 \\ 3 & 4 & 2}$};
      \node at (0,2) (A2) {$\vline\smtab{& & 2 \\ \\ 5 \\ 4 & 3 & 1}$};
      \node at (2,2) (C2) {$\vline\smtab{& & 3 \\ \\ 2 \\ 4 & 5 & 1}$};
      \node at (4,2) (E2) {$\vline\smtab{& & 4 \\ \\ 1 \\ 3 & 5 & 2}$};
      \node at (1,1) (B1) {$\vline\smtab{& & 2 \\ \\ 3 \\ 4 & 5 & 1}$};
      \node at (3,1) (D1) {$\vline\smtab{& & 3 \\ \\ 1 \\ 4 & 5 & 2}$};
      \draw[thin,label={[above]{$\swap_4$}}] (C5) -- (B4) ;
      \draw[thin,label={[above]{$\braid_3$}}](C5) -- (D4) ;
      \draw[thin,label={[above]{$\swap_3$}}] (B4) -- (A3) ;
      \draw[thin,label={[above]{$\swap_4$}}] (D4) -- (C3) ;
      \draw[thin,label={[above]{$\swap_1$}}] (D4) -- (E3) ;
      \draw[thin,label={[above]{$\swap_2$}}] (A3) -- (A2) ;
      \draw[thin,label={[above]{$\swap_3$}}] (C3) -- (C2) ;
      \draw[thin,label={[above]{$\swap_1$}}] (C3) -- (E2) ;
      \draw[thin,label={[above]{$\swap_4$}}] (E3) -- (E2) ;
      \draw[thin,label={[above]{$\braid_4$}}](A2) -- (B1) ;
      \draw[thin,label={[above]{$\swap_2$}}] (C2) -- (B1) ;
      \draw[thin,label={[above]{$\swap_1$}}] (C2) -- (D1) ;
      \draw[thin,label={[above]{$\swap_3$}}] (E2) -- (D1) ;
    \end{tikzpicture}
  \caption{\label{fig:bal-relations}An illustration of the Coxeter moves on $\Bal(42153)$.}
  \end{center}
\end{figure}

For examples of Coxeter moves on balanced tableaux, see Fig.~\ref{fig:bal-relations}. Comparing this with Fig.~\ref{fig:rex-relations} suggests a poset-preserving bijection between reduced words and balanced tableaux, and indeed we will demonstrate this bijection below.

\begin{lemma}
  The maps $\swap_i$ and $\braid_i$ are well-defined involutions on $\Bal(w)$.
  \label{lem:bal-down}
\end{lemma}

\begin{proof}
  For $R \in \Bal(w)$, if $i$ and $i+1$ are not in the same row or same column, then interchanging them cannot unbalance the tableau since all other entries compare the same with $i$ and with $i+1$. Thus $\swap_i(R) \in \Bal(w)$. If $i\pm 1$ is in the same row as $i$ and $i \mp 1$ is in the same column, then swapping them maintains the balance since, again, every $j$ less than $i-1$ or greater than $i+1$ compares with same with both, the two cannot be in the same row or same column as one another, and $i$ has traded the two to maintain its balance.
\end{proof}

\begin{remark}
  When $w$ is a permutation with a unique descent, $\D(w)$ has the form of the Young diagram (in English notation) for a partition, and standard balanced tableaux for $w$ are precisely the standard reverse Young tableau. In this case, the poset structure on $\Bal(w)$ where we consider only the \emph{Coxeter--Knuth relations} coincides with the dual equivalence graph \cite{Ass15} on standard reverse Young tableaux. For details on this connection and its combinatorial consequences, see \cite{Ass-W}.
\end{remark}

Parallel to the case of reduced words, we introduce a simple statistic on standard balanced tableaux that gives the minimum distance from a standard balanced tableau to the super-Yamanouchi one.

\begin{definition}
  For $R \in \Bal(w)$, the \newword{inversion number of $R$} is
  \begin{equation}
    \inv(R) = \#\{ (i<j) \mid i \mbox{ lies in strictly higher row, different column than } j \} .
    \label{e:bal-inv1}
  \end{equation}
  We call such a pair an \newword{inversion} of $R$.
  \label{def:bal-inv}
\end{definition}

\begin{example}[Inversion number of balanced tableaux]
  The standard balanced tableau in Fig.~\ref{fig:bal-inv} has $11$ inversion pairs as listed to the right. Notice that $(6,9)$ and $(4,8)$ are \emph{not} inversions since these pairs occur in the same column.
  \label{ex:inversion-B}
\end{example}

\begin{figure}[ht]
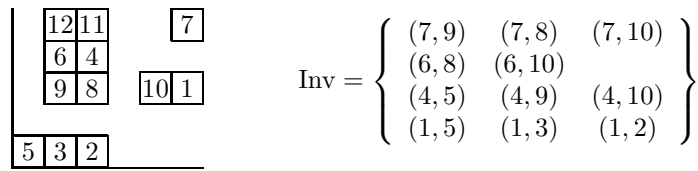

  \begin{displaymath}
    \vline\tableau{ & 12 & 11 & & & 7 \\ & 6 & 4 \\ & 9 & 8 & & 10 & 1 \\ \\ 5 & 3 & 2 \\\hline }
    \hspace{3\cellsize}
    \raisebox{-1.5\cellsize}{$\mathrm{Inv} = 
    \left\{ \begin{array}{ccc}
      (7,9) & (7,8) & (7,10) \\
      (6,8) & (6,10) & \\
      (4,5) & (4,9) & (4,10) \\
      (1,5) & (1,3) & (1,2) 
    \end{array} \right\}$}
  \end{displaymath}
  \caption{\label{fig:bal-inv}The inversion pairs for a standard balanced tableau.}
\end{figure}

\begin{theorem}
  Let $P_w \in \Bal(w)$ be the unique super-Yamanouchi tableau. Then for any $R \in \Bal(w)$, there exists a sequence $f = f_{\inv(R)} \cdots f_1$ of Coxeter moves such that $f(P_w) = R$, and, for any sequence $g = g_m \cdots g_1$ of Coxeter moves with $g(P_w)=R$, we have $m \geq \inv(R)$.
  \label{thm:bal-super}
\end{theorem}

\begin{proof}
  We proceed by induction on $\inv(R)$. Clearly $\inv(P_w)=0$ since it is the unique balanced filling such that all larger entries occur weakly above smaller entries, and the result holds for this case. Moreover, if $R$ has some $i<j$ with $i$ above $j$ and in the same column, then the balanced condition ensures that there is some $k>j$ in the same row as $j$, and so $i<k$ with $i$ and $k$ not in the same column. In particular, $\inv(R) > 0$ for $R \neq P_w$. This establishes the base case.
  
  Let $R \in \Bal(w)$ with $\inv(R)>0$. We claim that there is a pair $(i,i+1)$ with $i$ above $i+1$. If not, then for any pair $(i<j)$ with $i$ above $j$ (such a pair exists since $\inv(R)>0$), there exists $k$ with $i<k<j$ and neither $(i<k)$ nor $(k<j)$ has the smaller strictly above the larger. Thus $k$ is weakly above $i$ and weakly below $j$, an impossibility since $i$ is strictly above $j$. Therefore we may take $i$ such that $i+1$ lies in a strictly lower row. There are two cases to consider.

  If $i$ and $i+1$ are not in the same column, then $\swap_i$ acts non-trivially on $R$. Furthermore, $\inv(\swap_i(R)) = \inv(R)-1$ since the pair $(i,i+1)$ is removed from the set of inversions and all other pairs remain but with $i$ and $i+1$ interchanged. By induction, the result holds for $\swap_i(R)$, and so, too, for $R$. 

  If $i$ and $i+1$ are in the same column for every pair with $i$ above $i+1$, then take $i$ maximal among all such pairs. We claim that $i+2$ must lie in the same row and to the right of $i+1$. If not, then $i+2$ must lie strictly above $i+1$, and, by the choice of $i$, $k+1$ must lie weakly above $k$ for all $k>i+2$. However, this would mean no larger entry was in the row of $i+1$, contracting the balanced condition since $i$ is in the same column and above it. Therefore $i+2$ does lie in the same row as $i+1$, and so $\braid_{i+1}$ acts non-trivially on $R$ by interchanging $i$ and $i+2$. Furthermore, $\inv(\braid_{i+1}(R)) = \inv(R)-1$ since the pair $(i,i+2)$ is removed from the set of inversions and all other pairs remain but with $i$ and $i+2$ interchanged. By induction, the result holds for $\braid_{i+1}(R)$, and so it holds for $R$ as well.   
\end{proof}

Parallel to Corollary~\ref{cor:yang-baxter}, we can also refine our calculation of Coxeter distance to count only the number of Yang--Baxter moves by considering \emph{column inversions}.

\begin{corollary}
  For $R \in \Bal(w)$, and $f = f_{k} \cdots f_1$ any minimal length sequence of Coxeter moves, i.e. $f_j = \swap_i$ or $\braid_i$, such that $f(R)$ is super-Yamanouchi, the number of Coxeter moves that are Yang--Baxter moves is equal to the number of column inversions of $R$, i.e.
  \begin{equation}
    \#\{ j \mid f_j = \braid_i \mbox{ some } i\} = \#\{ (i<j) \mid i \mbox{ in higher row, same column as } j \} . \vspace{-\baselineskip}
  \end{equation}
  \label{cor:yang-baxter-bal}
\end{corollary}

Computing the \emph{permutation} of a balanced tableau is also far simpler.

\begin{definition}
  Given $R \in \Bal(w)$, define the \newword{permutation of $R$}, denoted by $\perm(R)$, by sorting the rows of $R$ to be decreasing (read left to right) and taking the reverse row reading word of the result.
  \label{def:bal-perm}
\end{definition}

\begin{example}[Permutation of balanced tableaux]
  Letting $R$ be the balanced tableau in Fig.~\ref{fig:bal-perm}, we have $\perm(R) = 2 \, 3 \, 5 \, 1 \, 8 \, 9 \, 1\!0 \, 4 \, 6 \, 7 \, 1\!1 \, 1\!2$.

  \begin{figure}[ht]
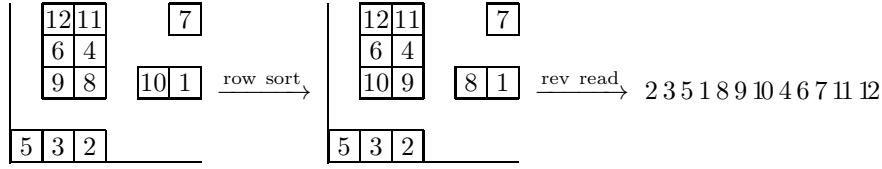

    \begin{displaymath}
      \vline\tableau{ & 12 & 11 & & & 7 \\ & 6 & 4 \\ & 9 & 8 & & 10 & 1 \\ \\ 5 & 3 & 2 \\\hline }
      \hspace{0.5\cellsize}\raisebox{-2\cellsize}{$\displaystyle\xrightarrow{\mathrm{row \ sort}}$}\hspace{0.5\cellsize}
      \vline\tableau{ & 12 & 11 & & & 7 \\ & 6 & 4 \\ & 10 & 9 & & 8 & 1 \\ \\ 5 & 3 & 2 \\\hline }
      \hspace{0.5\cellsize}\raisebox{-2\cellsize}{$\displaystyle\xrightarrow{\mathrm{rev \ read}}$}\hspace{0.5\cellsize}
      \raisebox{-2\cellsize}{$2 \, 3 \, 5 \, 1 \, 8 \, 9 \, 1\!0 \, 4 \, 6 \, 7 \, 1\!1 \, 1\!2$}
    \end{displaymath}
    \caption{\label{fig:bal-perm}Constructing the permutation of a standard balanced tableau.}
  \end{figure}

  Note that while $R$ has $11$ inversions, its associated permutation has length $13$. The difference is precisely the number of steps needed to sort the rows of the tableau. Moreover, letting $P$ be the super-Yamanouchi filling, we have
    \[ R = \swap_7 \, \swap_8 \, \swap_9 \, \swap_4 \, \swap_6 \, \braid_8 \, \braid_6 \, \swap_7 \, \swap_1 \, \swap_2 \, \swap_3 \, P, \]
  which is a sequence of $11$ involutions, exactly $2$ of which are Yang--Baxter moves.
\end{example}

\begin{theorem}
  For $R\in\Bal(w)$, we have
  \begin{equation}
    \inv(R) = \ell(\perm(R)) - \sum_r \mathrm{coinv}(\row_r(R)), \vspace{-0.5\baselineskip}
    \label{e:bal-inv2}
  \end{equation}
  where $\mathrm{coinv}(\row_r(R))$ is the number of entries $i<j$ with $i$ left of $j$ in row $r$.
  \label{thm:bal-inv}
\end{theorem}

\begin{proof}
  Let $\mathrm{I}$ be defined by the right hand side of \eqref{e:bal-inv2}. Let $R\in \Bal(w)$ and suppose $\swap_i$ acts non-trivially on $R$. Then $i$ and $i+1$ lie in different rows and different columns in $R$, so $\mathrm{sort}(R)$ and $\mathrm{sort}(\swap_i R)$ differ exactly in that $i$ and $i+1$ have been exchanged, and so $\perm(\swap_i R) = s_{i} \perm(R)$. Further, since all letters other than $i,i+1$ compare the same with $i$ and $i+1$, $R$ and $\swap_i R$ have the same number of row (co)inversions. In particular, we have
  \[ \mathrm{I}(\swap_i R) = \ell(s_{i} \perm(R)) - \sum_r \mathrm{coinv}(\row_r(R)) = \mathrm{I}(R) \pm 1, \vspace{-0.5\baselineskip}\]
  and, moreover, $\mathrm{I}(\swap_i R) = \mathrm{I}(R)+1$ precisely when $i$ is left of $i+1$ in $v$.

  Next suppose that $\braid_i$ acts non-trivially on $R$, exchanging $i-1$ and $i+1$ when $i$ lies directly below the one and directly left of the other. The permutation exchanging $i-1$ and $i+1$ is given by $s_{i-1}s_is_{i-1} = s_i s_{i-1} s_i$, but since $i-1$ and $i+1$ compare differently with $i$, when the rows are sorted the one in the row of $i$ will flip to the other side of it. Therefore $\perm(\braid_i R) = s_{i-1} s_{i} \perm(R)$ if $i+1$ is above $i-1$, and $\perm(\braid_i R) = s_{i} s_{i-1} \perm(R)$ otherwise, and in the former case we have
  \[\mathrm{I}(\braid_i R) = \ell(s_{i-1} s_{i} \perm(R)) - \sum_r\left(\mathrm{coinv}(\row_r(R)) + 1\right) = \mathrm{I}(R) + 1,\vspace{-0.5\baselineskip} \]
  and, by the same computation, $\mathrm{I}(\braid_i R) = \mathrm{I}(R)-1$ in the latter case.
  
  By Theorem~\ref{thm:bal-super}, $\inv(R)=0$ if and only if $R$ is super-Yamanouchi, in which case $\perm(R)$ is the identity and $R$ has decreasing rows, thus giving $\mathrm{I}(R)=0$ as well. Conversely, if we consider $\hat{v}$ to be the permutation obtained by following Definition~\ref{def:bal-perm} without first sorting the rows of $R$, then we have $\ell(\hat{v}) = \ell(v) + \sum_r \mathrm{coinv}(\row_r(R))$. In particular, $\mathrm{I}(R)=0$ if and only if $v$ is the identity, in which case $R$ is super-Yamanouchi. Therefore $\inv(R) = \mathrm{I}(R)$ whenever either is $0$. By Theorem~\ref{thm:bal-super}, for any $R\in\Bal(w)$, we may write $R = f_{\inv(R)} \cdots f_1(P)$, where $P$ is super-Yamanouchi and each $f_i$ is a Coxeter move. The result for $R$ now follows from the analysis of Coxeter moves above.
\end{proof}

Comparing Theorem~\ref{thm:rex-super} with Theorem~\ref{thm:bal-super}, one can anticipate the bijection between reduced words and standard balanced tableaux preserves the permutation and inversion number. Indeed, given the permutation $v$, one can recover the row entries for the corresponding balanced tableau, if it exists. The following result shows there is at most one balanced tableau with the given row entries.

\begin{lemma}
  For $R,S \in \Bal(w)$, if $R$ and $S$ both row sort to $T$, then $R=S$.
\end{lemma}

\begin{proof}
  We will show there is at most one ordering on the rows of a filling $T$ such that $T$ is balanced. Beginning with the top row, we must place entries in decreasing order from left to right. Assuming all higher rows have been uniquely balanced, begin balancing row $r$ from left to right. If the available entries for cell $x$ are $x_1 > \cdots > x_k$, then let $c_i$ be the number of cells above $x$ that are smaller than $x_i$, and let $r_i = i-1$, which is the number of entries right of $x$ that will be greater than $x_i$ should it be placed into cell $x$. Note that $c_1 \geq \cdots \geq c_k$ and $r_1 < \cdots < r_k$. Thus there is at most one index $i$ for which $r_i=c_i$, i.e. there is at most one entry that can be placed into cell $x$ for which the resulting tableau will be balanced.
\end{proof}

We have now established the following isomorphism of posets.

\begin{theorem}
  We have a poset isomorphism $\varphi : \Red(w) \stackrel{\sim}{\rightarrow} \Bal(w)$ such that $\varphi(\rho) = R$ if and only if $\perm(\rho) = \perm(R)$. In particular, $\varphi$ preserves the rank.
\end{theorem}

\begin{example}[Poset isomorphism $\Red(w) \stackrel{\sim}{\rightarrow} \Bal(w)$]
  The running examples in $\Red(w)$ and $\Bal(w)$ for the permutation $w = 41758236$ both have associated permutation $2 \, 3 \, 5 \, 1 \, 8 \, 9 \, 1\!0 \, 4 \, 6 \, 7 \, 1\!1 \, 1\!2$, and so correspond under the bijection. 
\end{example}

As a consequence, we recover the following result of Edelman and Greene \cite{EG87}, also proved bijectively by Fomin, Greene, Reiner, and Shimozono \cite{FGRS97}.

\begin{corollary}
  The number of reduced words for $w$ is equal to the number of standard balanced tableaux of shape $\D(w)$.
  \label{cor:rex-bal}
\end{corollary}

%
\section{Involutions and the long permutation}
%
\label{sec:long}

It is easy to see that if $\rho$ is a reduced word for $w$, then the reversal of $\rho$ is a reduced word for $w^{-1}$. We give the analogous involution on balanced tableaux.

\begin{definition}
  Define the \newword{flip map} $\varphi$ on standard balanced tableaux by setting $\varphi(R)$ to be the transpose of $R$ composed with replacing entry $i$ with $\ell-i+1$, where $\ell$ is the number of cells of $R$. 
\end{definition}

\begin{example}[Flip map]
  The flip map applied to $R \in \Bal(41758236)$ from Example~\ref{ex:inversion-B} results in $\varphi(R) \in \Bal(26714835)$ shown in Figure~\ref{fig:flip}. As $R$ corresponds to $\rho = (5,6,3,4,5,7,3,1,4,2,3,6)$ in Example~\ref{ex:inversion-R}, we may also consider the reversal of $\rho$ given by $\mathrm{rev}(\rho) = (6,3,2,4,1,3,7,5,4,3,6,5)$. We can easily compute
  \[ \perm(\varphi(R)) = 8 \, 1 \, 4 \, 7 \, 1\!0 \, 2 \, 5 \, 9 \, 1\!1 \, 3 \, 6 \, 1\!2 \]
  from Figure~\ref{fig:flip}, and less easily compute by Algorithm~\ref{alg:rex-perm} that this coincides with $\perm(\mathrm{rev}(\rho))$, indicating that $\varphi(R)$ corresponds to $\mathrm{rev}(\rho)$.
\end{example}

\begin{figure}[ht]
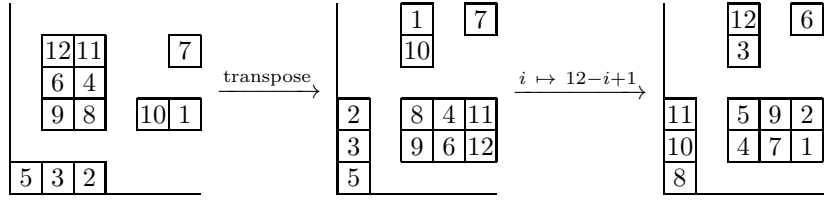

  \begin{displaymath}
    \vline\tableau{ \\ & 12 & 11 & & & 7 \\ & 6 & 4 \\ & 9 & 8 & & 10 & 1 \\ \\ 5 & 3 & 2 \\\hline }
    \hspace{0.5\cellsize}\raisebox{-2\cellsize}{$\displaystyle\xrightarrow{\mathrm{transpose}}$}\hspace{0.5\cellsize}
      \vline\tableau{ & & 1 & & 7 \\ & & 10 \\ \\ 2 & & 8 & 4 & 11 \\ 3 & & 9 & 6 & 12 \\ 5 \\\hline }
      \hspace{0.5\cellsize}\raisebox{-2\cellsize}{$\displaystyle\xrightarrow{i \ \mapsto \ 12-i+1}$}\hspace{0.5\cellsize}
      \vline\tableau{ & & 12 & & 6 \\ & & 3 \\ \\ 11 & & 5 & 9 & 2 \\ 10 & & 4 & 7 & 1 \\ 8 \\\hline }
  \end{displaymath}
  \caption{\label{fig:flip}The flip map applied to a standard balanced tableau.}
\end{figure}

\begin{proposition}
  The flip map $\varphi$ is a well-defined involution that maps $\Bal(w)$ to $\Bal(w^{-1})$ such that $\varphi(\swap_i(R)) = \swap_{\ell-i}(\varphi(R))$ and $\varphi(\braid_i(R)) = \braid_{\ell-i+1}(\varphi(R))$. 
  \label{prop:flip}
\end{proposition}

\begin{proof}
  By \eqref{e:rothe}, the Rothe diagram for $w^{-1}$ is the transpose of the Rothe diagram for $w$, and so the flip map $\varphi$ is a well-defined from $\Bal(w)$ to $\Bal(w^{-1})$ if its image is balanced. A filling $R$ is balanced if and only if for each cell $y$ of $R$ we have
  \[ \#\{ x \in R \mid x < y \mbox{ and } x \mbox{ above } y\} = \#\{ z \in R \mid z > y \mbox{ and } z \mbox{ right of } y\}, \]
  where $x$ is in the same column and $z$ is in the same row as $y$. Transposing $R$ to $R^{T}$ results in a filling such that each cell $y$ satisfies
  \[ \#\{ x \in R^{T} \mid x < y \mbox{ and } x \mbox{ right of } y\} = \#\{ z \in R^{T} \mid z > y \mbox{ and } z \mbox{ above } y\}, \]
  where $x$ is now in the row of $y$ and $z$ is in the column of $y$. Replacing $i$ with $\ell-i+1$ reverses the relative order of entries, so that each cell $y$, we have
  \[ \#\{ x \in \varphi(R) \mid x > y \mbox{ and } x \mbox{ right of } y\} = \#\{ z \in \varphi(R) \mid z < y \mbox{ and } z \mbox{ above } y\}, \]
  where $x$ is in the row of $y$ and $z$ is in the column of $y$, i.e. $\varphi(R)$ is balanced. 

  Since $i$ and $i+1$ are not in the row or column in $R$ if and only if $\ell-i+1$ and $\ell-i$ are not in the row or column in $\varphi(R)$, we have $\varphi(\swap_i(R)) = \swap_{\ell-i}(\varphi(R))$. Similarly, $i-1,i,i+1$ form a braid pattern in $R$ if and only if $\ell-i+2,\ell-i+1,\ell-i$ form a braid pattern in $\varphi(R)$, showing $\varphi(\braid_i(R)) = \braid_{\ell-i+1}(\varphi(R))$.
\end{proof}

Using the ranked poset structure on reduced words and balanced tableaux and the observation that $\mathrm{rev}(\swap_i(\rho)) = \swap_{\ell-i}(\mathrm{rev}(\rho))$ and $\mathrm{rev}(\braid_i(\rho)) = \braid_{\ell-i+1}(\mathrm{rev}(\rho))$, we have the following equivalence of involutions.

\begin{corollary}
  Given a permutation $w$, if $R\in\Bal(w)$ corresponds to $\rho\in\Red(w)$, then $\varphi(R)\in\Bal(w^{-1})$ corresponds to $\mathrm{rev}(\rho)\in\Red(w^{-1})$.
\end{corollary}

While these involutions respect the graph structure on reduced words and balanced tableaux, they do not behave particularly well with respect to the ranking. When $w$ is particularly nice, or rather, when the Rothe diagram of $w$ is particularly nice, there is a different involution that does respect the poset structure.

\begin{theorem}
  For the longest permutation $w_0^{(n)} = n (n-1) \cdots 2 1$ of $\mathfrak{S}_n$, the map $\psi$ sending an entry $i$ to $\binom{n}{2}-i+1$ is an order-reserving involution on $\Bal(w_0^{(n)})$. In particular, $\Bal(w_0^{(n)})$ has a unique minimal element $B$ with
  \[ \inv(B) = \frac{(n-2)(n-1)(n)(3n-5)}{24} . \]
  \label{thm:reverse}
\end{theorem}

\begin{proof}
  The Rothe diagram $\D(w_0^{(n)})$ is the staircase diagram $\delta_{n-1}$ of left-justified rows of lengths $1,2,\ldots,n-1$ from top to bottom. Thus every cell $y$ of $\D(w_0^{(n)})$ has as many cells above it as to its right. For $y$ a cell of $\D(w_0^{(n)})$, let $\mathrm{leg}(y)$ denote the set of cells above $y$ in the same column and let $\mathrm{arm}(y)$ denote the set of cells to the right of $y$ in the same row. Then
  \begin{eqnarray*}
    \#\{ x \in \mathrm{leg}(y) \mid x < y \} & = & \#\mathrm{leg}(y) - \#\{ x \in \mathrm{leg}(y) \mid x > y \}, \\
    \#\{ z \in \mathrm{arm}(y) \mid z > y \} & = & \#\mathrm{arm}(y) - \#\{ z \in \mathrm{arm}(y) \mid z < y \}.
  \end{eqnarray*}
  For $R \in \Bal(w_0^{(n)})$, since $\#\mathrm{leg}(y) = \#\mathrm{arm}(y)$ for every $y$, this implies 
  \begin{eqnarray*}
    \#\{ x \in \mathrm{leg}(y) \mid x > y \} & = & \#\{ z \in \mathrm{arm}(y) \mid z < y \},
  \end{eqnarray*}
  from which it follows that $\psi(R)$ is balanced.

  For every pair of cells $x,y$ neither in the same row nor same column, say with $x$ above $y$, the pair $(x,y)$ is an inversion in $R$ if and only if it is not an inversion in $\psi(R)$. In particular, every such pair is an inversion only for $\psi(P)$, where $P$ is the super-Yamanouchi tableau. To compute the number of such pairs, notice that there are $\binom{k}{2}$ cells above the cell in the $k$th row from the top, and we should not have counted $k-1$ cells in the first column, $k-2$ in the second, and so on, giving 
  \[ \sum_{k=1}^{n-1} k \binom{k}{2} - \sum_{k=1}^{n-1} \binom{k}{2}
  = \frac{1}{4}(3n-1)\binom{n}{3} - \binom{n}{3}
  = \frac{(n-2)(n-1)(n)(3n-5)}{24}, \]
  where the leftmost summation is the (signless) Stirling numbers of the first kind $\mathrm{s}(n,n-2)$ and the rightmost is the tetrahedral numbers.
\end{proof}

\begin{example}[Minimal element of $\Bal(w_0^{(n)})$]
  The ranked poset on $\Bal(w_0^{(4)})$ is shown in Fig.~\ref{fig:B321}. Notice that the unique minimal element is given by
  \[ \psi\left(\raisebox{0.5\cellsize}{$\smtab{6 \\ 5 & 4 \\ 3 & 2 & 1}$}\right)
  \ = \
  \raisebox{0.5\cellsize}{$\smtab{1 \\ 2 & 3 \\ 4 & 5 & 6}$} \]
  and the number of inversions for this minimum is $11 - 4 = 7$.
\end{example}

\begin{figure}[ht]
  \begin{center}
    \begin{tikzpicture}[xscale=1.6,yscale=1,
        label/.style={%
          postaction={ decorate,
            decoration={ markings, mark=at position 0.75 with \node #1;}}}]
      \node at (3,7) (R37) {$\smtab{6 \\ 5 & 4 \\ 3 & 2 & 1}$};
      \node at (2,6) (R26) {$\smtab{4 \\ 5 & 6 \\ 3 & 2 & 1}$};
      \node at (4,6) (R46) {$\smtab{6 \\ 5 & 3 \\ 4 & 2 & 1}$};
      \node at (1,5) (R15) {$\smtab{2 \\ 5 & 6 \\ 3 & 4 & 1}$};
      \node at (5,5) (R55) {$\smtab{6 \\ 5 & 1 \\ 4 & 2 & 3}$};
      \node at (0,4) (R04) {$\smtab{1 \\ 5 & 6 \\ 3 & 4 & 2}$};
      \node at (2,4) (R24) {$\smtab{2 \\ 4 & 6 \\ 3 & 5 & 1}$};
      \node at (6,4) (R64) {$\smtab{6 \\ 3 & 1 \\ 4 & 2 & 5}$};
      \node at (1,3) (R13) {$\smtab{1 \\ 4 & 6 \\ 3 & 5 & 2}$};
      \node at (5,3) (R53) {$\smtab{5 \\ 3 & 1 \\ 4 & 2 & 6}$};
      \node at (7,3) (R73) {$\smtab{6 \\ 2 & 1 \\ 4 & 3 & 5}$};
      \node at (2,2) (R22) {$\smtab{1 \\ 2 & 6 \\ 3 & 5 & 4}$};
      \node at (6,2) (R62) {$\smtab{5 \\ 2 & 1 \\ 4 & 3 & 6}$};
      \node at (3,1) (R31) {$\smtab{1 \\ 2 & 4 \\ 3 & 5 & 6}$};
      \node at (5,1) (R51) {$\smtab{3 \\ 2 & 1 \\ 4 & 5 & 6}$};
      \node at (4,0) (R40) {$\smtab{1 \\ 2 & 3 \\ 4 & 5 & 6}$};
      \draw[thin,label={[above]{$\braid_5$}}](R37) -- (R26) ;
      \draw[thin,label={[above]{$\swap_3$}}] (R37) -- (R46) ;
      \draw[thin,label={[above]{$\braid_3$}}](R26) -- (R15) ;
      \draw[thin,label={[above]{$\braid_2$}}](R46) -- (R55) ;
      \draw[thin,label={[above]{$\swap_1$}}] (R15) -- (R04) ;
      \draw[thin,label={[above]{$\swap_4$}}] (R15) -- (R24) ;
      \draw[thin,label={[above]{$\braid_4$}}](R55) -- (R64) ;
      \draw[thin,label={[above]{$\swap_4$}}] (R04) -- (R13) ;
      \draw[thin,label={[above]{$\swap_1$}}] (R24) -- (R13) ;
      \draw[thin,label={[above]{$\swap_5$}}] (R64) -- (R53) ;
      \draw[thin,label={[above]{$\swap_2$}}] (R64) -- (R73) ;
      \draw[thin,label={[above]{$\braid_3$}}](R13) -- (R22) ;
      \draw[thin,label={[above]{$\swap_2$}}] (R53) -- (R62) ;
      \draw[thin,label={[above]{$\swap_5$}}] (R73) -- (R62) ;
      \draw[thin,label={[above]{$\braid_5$}}](R22) -- (R31) ;
      \draw[thin,label={[above]{$\braid_4$}}](R62) -- (R51) ;
      \draw[thin,label={[above]{$\swap_3$}}] (R31) -- (R40) ;
      \draw[thin,label={[above]{$\braid_2$}}](R51) -- (R40) ;
    \end{tikzpicture}
  \caption{\label{fig:B321}An illustration of the Coxeter moves on $\Bal(4321)$.}
  \end{center}
\end{figure}

The graph on reduced words for $w_0^{(n)}$ is of particular interest. Dehornoy and Autord \cite{DA10} proved that the diameter of the graph for $w_0^{(n)}$ grows asymptotically like $n^4$. Reiner and Roichman \cite{RR13} used hyperplane arrangements to prove an exact formula for the diameter that coincides with $\inv(B)$ in Theorem~\ref{thm:reverse}. We give a new, elementary proof using the inversion metric on balanced tableaux.

\begin{corollary}
  The maximum distance between two reduced words for $w_0^{(n)}$ is
  \begin{equation}
    \max_{\rho,\sigma\in\Red(w_0^{(n)})} \mathrm{dist}(R,S) = \frac{(n-2)(n-1)(n)(3n-5)}{24} .
  \end{equation}
\end{corollary}

\begin{proof}
  Let $P$ denote the super-Yamanouchi balanced tableau for $w_0^{(n)}$, and let $B = \psi(P)$. Given any balanced tableau $R \in \Bal(w_0^{(n)})$, there is an $\inv$-increasing path from $P$ to $R$ and, by considering the reversed poset assured by Theorem~\ref{thm:reverse}, an $\inv$-decreasing path from $R$ to $B$. Therefore we have
  \begin{equation}
    \mathrm{dist}(P,R) + \mathrm{dist}(R,B) = \mathrm{dist}(P,B).
    \label{e:dist}
  \end{equation}
  For $R,S \in \Bal(w_0^{(n)})$, the triangle inequality gives
  \[ \mathrm{dist}(R,P) + \mathrm{dist}(P,S) \geq \mathrm{dist}(R,S) \leq \mathrm{dist}(R,B) + \mathrm{dist}(B,S) . \]
  Combining this with Eq.~\ref{e:dist} for both $R$ and $S$, we have
  \[ 2 \, \mathrm{dist}(R,S) \leq \mathrm{dist}(R,P) + \mathrm{dist}(P,S) + \mathrm{dist}(R,B) + \mathrm{dist}(B,S)  = 2 \, \mathrm{dist}(P,B). \]
  Thus $\mathrm{dist}(R,S) \leq \mathrm{dist}(P,B) = \inv(B)$ for all $R,S \in \Bal(w_0^{(n)})$. In particular, the diameter of the graph is $\inv(B)$, so the result follows from Theorem~\ref{thm:reverse}.
\end{proof}

%
\section*{Acknowledgments}
%
\label{sec:thanks}

The author is grateful to Bridget Tenner and Vic Reiner for interesting discussions and helpful comments on early drafts. 

%
%

\bibliographystyle{amsplain} 
\bibliography{redinv}

\end{document}